\documentclass[conference]{IEEEtran}
\usepackage{mathrsfs}
\usepackage{amssymb}
\usepackage{amsmath}
\usepackage{mathrsfs}
\usepackage{graphicx}
\usepackage{eurosym}
\usepackage{amssymb}
\usepackage{amsmath}
\usepackage{amsfonts}
\usepackage{epstopdf}
\usepackage{epsf,subfigure}
\usepackage{psfrag}
\usepackage{graphics}
\usepackage{color} 
\usepackage{cite}
\usepackage{array,multirow,pbox}
\usepackage{enumitem}
\usepackage{textcomp}
\usepackage{gensymb}
\usepackage{tikz}
\usepackage[english]{babel}
\usepackage{caption}

\newcommand{\myprob}[2]{
\begin{enumerate}[label=\textbf{P\arabic*}]
\setcounter{enumi}{\value{#1}}
#2
\setcounter{#1}{\value{enumi}}
\end{enumerate}}

\newcommand{\myassum}[2]{
\begin{enumerate}[label=\textbf{A\arabic*}]
\setcounter{enumi}{\value{#1}}
#2
\setcounter{#1}{\value{enumi}}
\end{enumerate}}

\newcommand{\bfg}[1]{\boldsymbol{#1}}
\newcommand{\rem}[1]{}
\newtheorem{proposition}{Proposition}

\newtheorem{theorem}{Theorem}

\newtheorem{lemma}{Lemma}
\newtheorem{remark}{Remark}
\newtheorem{assumption}{Assumption}

\newenvironment{proof}[1][Proof]{\begin{trivlist}
\item[\hskip \labelsep {\bfseries #1}]}{\end{trivlist}}

\newcommand{\red}[1]{{\color{red} #1}}
\newcommand{\qed}{\nobreak \ifvmode \relax \else
      \ifdim\lastskip<1.5em \hskip-\lastskip
      \hskip1.5em plus0em minus0.5em \fi \nobreak
      \vrule height0.75em width0.5em depth0.25em\fi}

\graphicspath{{Figures/}}

\newcommand{\expect}[1]{\mathbb{E}\lbrace #1\rbrace}
\newcommand{\norm}[1]{\left\lVert#1\right\rVert}

\newcommand{\mysubeq}[2]{
\begin{subequations}\label{#1}
\begin{align}
#2
\end{align}\end{subequations}}

\ifCLASSINFOpdf
\else
\fi
\begin{document}
%
\title{Optimal Energy Consumption Forecast for Grid Responsive Buildings: A Sensitivity Analysis}

\author{\IEEEauthorblockN{Soumya Kundu, Thiagarajan Ramachandran, Yan Chen and Draguna Vrabie}
\IEEEauthorblockA{Electricity Infrastructure and Buildings Division\\
Pacific Northwest National Laboratory\\
Richland, WA 99352 USA\\
Email:\,\{soumya.kundu, thiagarajan.ramachandran, yan.chen, draguna.vrabie\}\,@pnnl.gov}}


%


\maketitle

\begin{abstract}
It is envisioned that building systems will become active participants in the smart grid operation by controlling their energy consumption to optimize complex criteria beyond ensuring local end-use comfort satisfaction. A forecast of the building energy consumption will be necessary to enable integration between building and grid operation. Such forecast will be affected by parametric and measurement uncertainty. In this paper we develop a methodology for quantifying the sensitivity of optimal hourly energy consumption forecasts to various sources model and measurement uncertainty. We demonstrate the approach for a building heating ventilation and air conditioning (HVAC) system use-case.
\end{abstract}


%
\IEEEpeerreviewmaketitle

\section{Introduction}
Commercial and residential buildings contribute to 72\% of all electricity consumption in the US. The transformation of the building systems from passive, albeit efficient, consumers to active participants in the smart grid operation constitutes a significant element towards improving grid resilience to extreme weather events and operation reliability in presence of large renewable integration. This can be accomplished by controlling and optimizing the individual building energy consumption, beyond simply ensuring end-use comfort satisfaction with minimal energy consumption. In a smart-grid operation context, residential and commercial buildings will be managing their energy consumption in accordance with some time-varying grid signals, by optimizing complex cost functions subject to satisfaction of end-use comfort constraints. 

Three principal questions need to be addressed in order to accomplish effective integrated operation between building energy systems and the power grid: (a) a priory evaluation of capability of a given building system to provide grid services, (b) control methodology development to enable building-grid coordinated operation (c) evaluation of energy consumption baseline associated with nominal operation of the building energy system. Numerous publications over the last decade are addressing these three questions with a strong focus on development of control concepts that enable integrated operation between building systems and the power grid. Without attempting to provide a comprehensive review we note few representative efforts. Methodology for assessing the capability of building energy systems to provide grid services has been proposed in \cite{Oldewurtel:2013}. Development and testing of control methods for ancillary services provision by commercial and residential building energy systems is reported in \cite{Lin_Barooah:2017} and \cite{Afshari:2017}. Methodology for calibrating baseline energy consumption models has been proposed in \cite{Coughlin:2009} and \cite{Sharifi:2016}. 

To ensure fair compensation for participation of intelligent building energy systems in the service markets, it is of critical importance to accurately forecast energy consumption baseline associated with nominal operation. This is consistent with current paradigms for grid operation and market designs, as well as proposed market designs for aggregated demand side flexibility as presented in \cite{Eid:2015}. In \cite{Mathieu:2011} the authors provide a detailed account of observed variability between the expected and measured building response. They note that while the variability occurs as a result of the combination of unmodeled internal and external load variability and equipment performance uncertainty, the largest amount of variability is due to baseline model error. This challenges the implementation of programs that use baselines for financial settlement such as demand/capacity bidding programs and ancillary services markets, or the Measurement and Verification efforts that are oriented towards evaluating the cost-effectiveness of building-grid integration.

Recognizing the effect of uncertainty in the baseline calculation, in this paper we develop a methodology for quantifying the sensitivity of optimal hourly energy consumption to various sources of parametric and measurement uncertainties. We demonstrate the developed mathematical approach for a commercial building HVAC system example. The approach enables quantitative assessment of the sensitivity of the estimated optimal baseline consumption under various operating conditions, both to parametric and measurement uncertainty. Herein we define the baseline as the optimal hourly energy consumption; this baseline consumption can be computed by solving a nonlinear optimization problem based on information of estimated load forecasts and parametric models describing the energy consumption performance of building equipment. Discrepancies in the sensors measurements, equipment models, or external or internal system loads (e.g. temperature, occupancy) will be reflected in the error associated with the estimated baseline consumption forecasts. 

The paper is organized as follows: in Section \ref{sec:Methodology} we develop the approach to evaluate optimal baseline performance and quantifying the sensitivity and calculating the error bars associated with uncertainties in the system model parameters and exogenous inputs. In Section \ref{sec:Baseline} we describe the specific use case example system and details the calculation of the baseline hourly energy consumption. In Section \ref{sec:Example} we provide the numerical results and capture the associated observations and insights.

\section{Sensitivity Analysis: Methodology}
\label{sec:Methodology}
The optimization problem to determine the baseline can be represented compactly as follows:
\mysubeq{E:optim}{
\text{(P0):}\quad\underset{x}{\text{minimize}}\quad& J(x,w)\\
\text{subject to}\quad& h(x,w)\leq 0\,,\,h:\mathbb{R}^m\!\times\!\mathbb{R}^p\!\mapsto\!\mathbb{R}^n\\
& x\in\mathbb{R}^m,\,w\in\mathbb{R}^p,\,
}
where $p$-dimensional vector $w$ denote the combination of exogenous inputs (forecast data, sensor measurements, thermal load profiles) and model parameters (both known and identified); $m$-dimensional vector $x$ denote the system variables, including the controller set-points; $J\!:\!\mathbb{R}^m\!\times\!\mathbb{R}^p\!\mapsto\!\mathbb{R}$ is the objective function (building energy consumption); while the feasibility domain is defined via a set of $n$ inequality constraints $h=[h_1,\, h_2,\, \dots,\,h_n]^T$. Note that any equality constraint can be represented by two inequality constraints.
\begin{assumption}\label{AS:poly}
$J,\,h$ and $g$ are \textit{analytic functions} \cite{Boas:1935}, i.e. the functions have continuous derivatives of \textit{all order} in $x$ and $w$.
\end{assumption}

We are interested in evaluating the performance and sensitivity of the optimal solution to uncertainties in the parameters and exogenous inputs (forecasts, sensors measurements). Let us model the (additive) uncertainties in the exogenous input and parameters as 
\begin{align}
w = w_0+\delta w
\end{align}
where $w_0$ denote the nominal (estimated) values and $\delta w$ are the bounded additive uncertainties around the nominal values. 

\subsection{Nominal Optimal Solution}

The optimization problem \eqref{E:optim} can be solved using the nominal values of the exogenous input and parameters, $w=w_0$\,, to find the optimal values of the system variables, $x=x_0$\,, that minimize the value of objective function:
\begin{align}
J_0:=J(x_0,w_0)\,.
\end{align}
The Karush-Kuhn-Tucker (KKT) conditions for the optimal solution are given by \cite{Bertsekas:1999}:
\mysubeq{E:KKT}{
\nabla_x J +{\sum}_{i=1}^n\lambda_i\,\nabla_x h_i&=0 \label{E:station}\\
\lambda_i\,h_i&=0~\forall i\in\lbrace 1,\dots,n\rbrace \label{E:ComSlac}\\
\lambda_i &\geq 0~\forall i\in\lbrace 1,\dots,n\rbrace\\
h_i &\leq 0~\forall i\in\lbrace 1,\dots,n\rbrace\,.
}
The non-negative scalars $\lambda=[\lambda_1,\lambda_2,\dots,\lambda_n]^T$ are the Lagrange multipliers. Note that the condition \eqref{E:ComSlac} (referred to as the `complementary slackness' condition) implies that either the inequality constraint is \textit{active} (or, \textit{binding}), i.e. $h_i=0$\,, or the corresponding Lagrange multiplier is zero. The condition \eqref{E:ComSlac} together with the `stationarity' condition \eqref{E:station} define the subspace
\begin{align}
\mathcal{X}_0:=\left\lbrace x\in\mathbb{R}^n\,\left|\begin{array}{r}
\nabla_xJ +{\sum}_{i=1}^n\lambda_i\,\nabla_x h_i=0\\
\forall i:~\lambda_i\,h_i=0\end{array}\right.\right\rbrace
\end{align}
which contains the nominal optimal solution, i.e. $x_0\in\mathcal{X}_0$\,. For brevity, let us introduce the notation
\begin{align}
H(x,w;\lambda):=\left[ \begin{array}{c}
\nabla_xJ +\lambda^T\,\nabla_x h\\
\lambda_1\,h_1\\
\vdots\\
\lambda_n\,h_n\end{array} \right]
\end{align}
such that, $\mathcal{X}_0=\left\lbrace x\in\mathbb{R}^n\,\left|\,H(x,w_0;\lambda)=0\right.\right\rbrace$\,.


\subsection{Sensitivity of Optimal Solution}

We are interested in estimating the changes in the optimal value of the cost function in response to small changes in the exogenous input and parameter values, $w=w_0+\delta w$\,. Specifically, we are interested in the following problem:
\mysubeq{E:optim1}{
\text{(P1):}\quad\underset{x}{\min}\quad& J(x,w_0+\delta w)\\
\text{s.\,t.\,}\quad& h(x,w_0+\delta w)\leq 0\,.
}
Solving the problem P1 \eqref{E:optim1} for every possible small values of $\left|\delta w\right|\leq \Delta$, for some given uncertainty bounds $\Delta\in\mathbb{R}_{\geq 0}^p$, is going to be computationally challenging, especially when the dimension of $w$ is large. Henceforth, we resort to solving an alternative problem which allows us to estimate the changes in optimal cost function without having to solve the P1 \eqref{E:optim1} for all possible uncertainty scenarios. In order to do so, let us make the following assumption:
\begin{assumption}\label{AS:H}
Changes in the exogenous input and parameter values, $\delta w=w-w_0$\,, are sufficiently small such that the optimal solution of P1 \eqref{E:optim1}, denoted by $x=x_0+\delta x$\,, lie on the subspace 
\begin{align*}
\mathcal{X}_1=\left\lbrace x\in\mathbb{R}^n\,\left|\,H(x,w_0+\delta w;\lambda)=0\right.\right\rbrace.
\end{align*}
\end{assumption}

The assumption says that the equations $H\!=\!0$ determine how the optimal values of the internal system variables $x$ change as the exogenous input and parameters $w$ experience small deviations around their nominal values. Note that the assumption is generally valid because of Assumption\,\ref{AS:poly}. Accordingly, we formulate an alternative to P1:
\mysubeq{E:optim2}{
\text{(P2):}\quad\underset{x}{\min}\quad& J(x,w_0+\delta w)\\
\text{s.\,t.\,}\quad& H(x,w_0+\delta w;\lambda)= 0\,,\label{E:optim2_c1}
}
where $E$ is a $\left|\mathcal{C}\right|\times m$ dimensional matrix of 0's and 1's which compactly represents the sets of equations $x_i=(x_0)_i~\forall i\in\mathcal{C}$\,, i.e. the first row of $E$ has a $1$ at the column indexed by the first element of $\mathcal{C}$ and 0's at all other colums\,, and likewise. Note that, when $H(x,w_0+\delta w;\lambda)=0$ has a unique solution in $x$\,, the problem P2 simply reduces to finding a solution to a set of (nonlinear) equations. Specifically, we assume the following:
\begin{assumption}\label{AS:unique}
There exists a unique $x$ satisfying \eqref{E:optim2_c1} over a range of bounded uncertainties $\left|\delta w\right|\!\leq\! \Delta$, for some $\Delta\!\in\!\mathbb{R}_{\geq 0}^p$.
\end{assumption} 

\begin{proposition}
The change in the optimal cost function value to (sufficiently) small changes in exogenous input and parameter values, $\delta w=w-w_0\,,$ can be approximated by\footnote{Note that $K(0)=0\,.$} 
\mysubeq{}{
K(\delta w)&:=J(x_0+G^+d\,,\,w_0+\delta w)-J_0\,,\\
\text{where, }~G^+&=\left(G^TG\right)^{-1}G^T\,,~G=\left.\nabla_x H\right|_{(x_0,w_0;\lambda)},\\
d&=-\left.\nabla_w H\right|_{(x_0,w_0;\lambda)}\cdot\delta w.
}
\end{proposition}
\begin{proof}
Note that, under the Assumption\,\ref{AS:H}, for small changes in $w=w_0+\delta w$\,, we can use Taylor series expansion to claim that the optimal solution $x$ would satisfy
\begin{align*}
\left.\nabla_x H\right|_{x_0,w_0;\lambda}\cdot (x-x_0)+\left.\nabla_w H\right|_{x_0,w_0;\lambda}\cdot \delta w &\approx 0\,,
\end{align*}
where we used $H(x_0,w_0)=0$\,. The above can be expressed in the compact form $G\cdot\left(x-x_0\right)=d$ which has a unique solution under Assumption\,\ref{AS:unique}. Specifically, the solution is unique if $G^TG$ is invertible, in which case $x-x_0=G^+d$\,, where $G^+=\left(G^TG\right)^{-1}G^T$ is the Moore-Penrose pseudo-inverse of the matrix $G$\,, \cite{Albert:1972}.
\hfill\hfill\qed\end{proof}
\begin{remark}
When Assumption\,\ref{AS:unique} does not hold, there is no unique solution to $G\cdot\left(x-x_0\right)=d$\,. In such cases, when $GG^T$ is invertible, there exist many solutions of the form $x-x_0=WG^T\left(GWG^T\right)^{-1}$ for every arbitrarily chosen positive definite matrix $W\in\mathbb{R}^{m\times m}$. While the optimal solution corresponds to some $W$\,, it is generally a non-convex problem to compute that.  
\end{remark}

Once the relationship between the change in input and parameters to the change in optimal cost function value is established, we proceed to computing the bounds on the change in cost function values for bounded uncertainties in the input and parameters. Specifically, in this paper, we focus our analysis to cases when the uncertainties in the input and parameters are bounded as follows,
\begin{align}
\left|\delta w\right|\leq \Delta
\end{align}
where $\Delta\in\mathbb{R}_{\geq 0}^p$ is a non-negative vector. We are interested in solving the following problem,
\mysubeq{E:optim_beta}{
\text{(P4):}~\,\text{minimize}& \quad\beta\\
\text{subject to} &~\left|K(\delta w)\right|\leq \beta \quad\forall \left\lbrace\delta w\left|\,~\left|\delta w\right|\leq \Delta\right.\right\rbrace.\label{E:Kw}
}

Solution of the problem P4 is non-trivial for generic forms of $K(\cdot)$\,. However, under special circumstances, it is possible to find either exact or an upper bound or an approximation of the solution to the optimization problem \eqref{E:optim_beta}. We present two approaches based on analytical and computational methods to find a solution to P4.

\subsubsection{Analytical Solution}

In the specific case when the function $K(\cdot)$ is at most a quadratic polynomial, it is possible to compute an upper bound of the solution to P4.
\begin{lemma}
If the polynomial $K(\cdot)$ has degree $\leq 2$\,, then 
\begin{align}
\beta\leq \norm{\nabla K(0)}_1\cdot\|\Delta\|_\infty+p\,\norm{M}_2^2\cdot\|\Delta\|_\infty^2
\end{align}
where $M^TM=\norm{\nabla^2K(0)}$\,; $\nabla K(0)$ and $\nabla^2 K(0)$ are the gradient and Hessian of $K(\cdot)$ evaluated at $\delta w=0$\,, while $\norm{\cdot}_1\,,\,\norm{\cdot}_2$ and $\norm{\cdot}_{\infty}$ denote the $l_1\,,\,l_2$ and $l_{\infty}$-norms, respectively.
\end{lemma}
\begin{proof}
Since $K(0)=0$\,, we can use Taylor series expansion to obtain
\begin{align*}
K(\delta w) = \nabla K(0)\cdot\delta w + (\delta w)^T\cdot\nabla^2K(0)\cdot \delta w
\end{align*}
Note that from H{\"o}lder's inequality\footnote{$|x^Ty|\leq |x|_n|y|_m$ for some $n,m$ satisfying $1/n+1/m=1$\,.} \cite{Hardy:1952} we have $\left|\nabla K(0)\,\delta w\right|\leq \norm{\nabla K(0)}_1\,\|\Delta\|_\infty $\,, while 
\begin{align*}
\left|(\delta w)^T\,\nabla^2K(0)\, \delta w\right|\leq (M\,\delta w)^T(M\,\delta w)\leq \norm{M}_2^2\norm{\delta w}_2^2
\end{align*}
where $M\in\mathbb{C}^{p\times p}$ is defined as the \textit{square-root} of the Hessian $\nabla^2K(0)$\,, i.e. $M^TM = \nabla^2K(0)$\,. The rest of the proof follows by applying the inequality $\norm{\delta w}_2\leq \sqrt{p}\,\norm{\delta w}_{\infty}$\,.
\hfill\hfill\qed\end{proof}

This result is applicable when one is interested in finding out the maximal deviation expected in the optimal cost function, given that maximal deviation in any input and parameter is within some known tolerance. However, other norm-bounded analysis can also be formulated using the same framework. 

\subsubsection{Computational Method}

When $K(\cdot)$ is a polynomial or a rational function, P4 can be solved using polynomial optimization techniques such as Sum-of-Squares (SOS) programming. As the name suggests, sum-of-squares (SOS) polynomials are any polynomial that can be expressed as a finite sum of squared polynomials. Using Putinar's Positivstellensatz (P-satz) theorem \cite{Putinar:1993,Lasserre:2009} (included here for completeness) we can translate semi-algebraic constraints of the form of \eqref{E:Kw} into SOS feasibility conditions. 

\begin{theorem}
[P-satz] Consider a compact domain $ \mathcal{K}\!=\!\!\left\lbrace x\!\in\!\mathbb{R}^n\!\left|\, g_i(x)\!\geq\! 0\,\forall i\!=\!1,\dots,m\right.\right\rbrace$, where $g_i(\cdot)\,\forall i$ are polynomials and $\left\lbrace x\left|\,g_i(x)\geq 0\right.\right\rbrace\,\forall i$ define compact domains. A polynomial $f(x)$ is positive on $\mathcal{K}$ if and only if there exist SOS polynomials $\sigma_i(x)$ such that $f(x) - \sum_i \sigma_i(x)g_i(x)$ is SOS.
\end{theorem} 

If $K(\cdot)$ is a polynomial, we formulate an SOS optimization problem as follows:
\begin{align}
\!\!\!\underset{S_1,S_2}{\min}&\quad\beta\\
\!\!\!\text{s.t.}&~\beta-K(\delta w) - S_{11}^T\left(\Delta\!-\!\delta w\right)-S_{12}^T\left(\Delta\!+\!\delta w\right)\text{ is SOS}\notag\\
\!\!\!&~\beta+K(\delta w) - S_{21}^T\left(\Delta\!-\!\delta w\right)-S_{22}^T\left(\Delta\!+\!\delta w\right)\text{ is SOS}\notag
\end{align}
which can be solved via semi-definite programming solvers (for details on SOS programming techniques, please refer to \cite{Parrilo:2000, sostools13, Sturm:1999}). Extension to rational function of the form $K(w)=K_1(w)/K_2(w)$ is possible by rewriting $\left|K(\delta w)\right|\leq \beta$ as $\left|K_1(\delta w)\right|\leq \beta\left|K_2(\delta w)\right|$.


\section{Baseline building hourly energy consumption forecast: an HVAC system use case}
\label{sec:Baseline}

In this section, we first describe the commercial HVAC system model, the optimal hourly energy consumption forecast problem, and the problem of analyzing sensitivity of the optimal solutions to input and parametric uncertainties. 

\subsection{Commercial Building HVAC System Model}\label{S:models}
We consider a commercial building HVAC system described as a single duct central air handling unit (AHU) with variable air volume (VAV) terminal boxes with reheat coils, serving the load characteristic to a small office building model from the DOE commercial building prototypes \cite{Thornton:2011,Goel:2014}; in a hot and dry climate where humidity effects can be ignored. The building has five zones with office and conference room utilization schedules. We utilized the EnergyPlus {\cite{Crawley:2000}} model to generate the hourly demand capacity for each building zone based on information on zone temperature set-point, occupancy profile, equipment schedules and weather data. The temperature set-point for each building zone is 22{\textdegree}C during the occupied period. 
The air distribution system is represented in Fig.\,\ref{fig:HVAC}. The AHU enables distribution of conditioned air to the buildings zones using a variable speed supply fan. The flow of supply air into the building zones is controlled by VAV dampers. Each VAV box includes a reheat coil (for additional heating) and a damper (to regulate air-flow rate). The AHU has heating and cooling coils that can enable regulation of the air temperature in the main distribution duct. Heating, reheating and cooling are provided by two hydronic systems served by a gas boiler and an absorption chiller. 
We utilized this model in our previous work \cite{Ramachandran:2017}.
\begin{figure}[ht]
\centering
  \includegraphics[width=\linewidth]{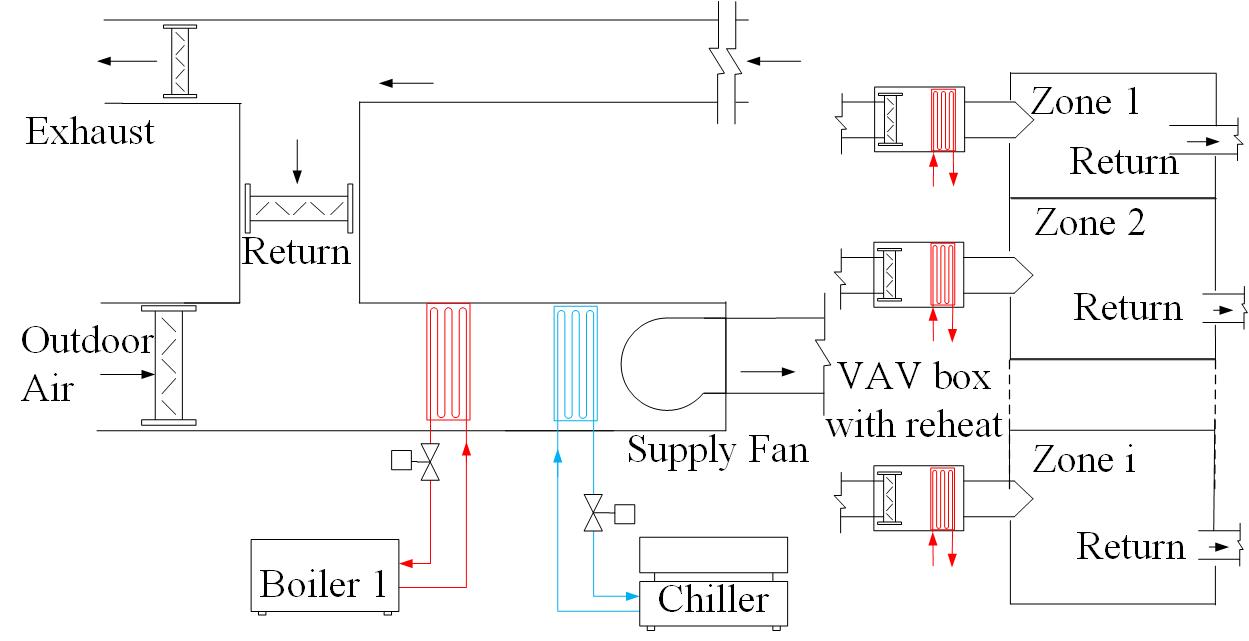}
  \caption{HVAC system schematic representation.}
  \label{fig:HVAC}
\end{figure}
%
The thermal load of each zone-$i$ can be described by the HVAC system operating set-points as:
\begin{align}
\label{eq:Q zone constraints}
\dot{Q}_{zone}^{i} = c_p \dot{m}_{sa}^{i} (T_{da}^{i} - T_{ra}^{i})\,,~\text{where }T_{ra}^i=T_{zone}^{SP,i}\,.
\end{align}
Here $c_p$ denotes the air specific heat capacity, $\dot{m}_{sa}^{i}$ denotes the mass air flow supplied to zone $i$, $T_{da}^{i}$ is the temperature of the air discharged in zone $i$, and $T_{ra}^{i}$ is the {air temperature returned in the AHU duct through re-circulation. We assume that the zone air temperature is maintained at a predefined set-point, denoted $T_{zone}^{SP,i}$, which is equal to the return air temperature}, i.e. $T_{ra}^{i} = T_{zone}^{SP,i}$\,. 
Assuming no air leakage in the supply duct
\begin{align}
\label{eq:msa constraints}
\dot{m}_{sa}\!=\!{\sum}_{i=1}^N \dot{m}_{sa}^{i}
\end{align}
where $\dot{m}_{sa}$ denotes the mass air-flow rate of the supply air to the building zones.
The following ensures that sufficient ventilation is maintained in the building zones:
\begin{align}\label{eq:m oa min}
\dot{m}_{sa}^{i} &\geq \dot{m}_{sa}^{i,min}:= \dfrac{\dot{m}_{sa}}{\dot{m}_{oa}}\,\dot{m}_{oa}^{i, min}\,,
\end{align}
where $\dot{m}_{oa}$ denotes the mass air-flow rate of the outside air; and $\dot{m}_{oa}^{i, min}$ denote the minimum required ventilation in each zone, determined by the available occupancy information. Assuming no heat exchange or air leakage in the return duct, the return air temperature $T_{ra}$ is given by
\begin{align}
\label{eq:Tra constraints}
T_{ra} = \dfrac{\sum_{i=1}^{N}\dot{m}_{sa}^{i}T_{ra}^{i}}{\dot{m}_{sa}}\,,
\end{align}
where {$N$ is the number of zones. Using three dampers in the AHU (see Fig.\,\ref{fig:HVAC}), a portion of the returned air can be mixed with outdoor air and recirculated and conditioned for redistribution in the zones. 
The mixed air temperature $T_{ma}$ is determined by
\mysubeq{}{
\label{eq:Tma constraints}
&T_{ma} \!=\! \dfrac{\dot{m}_{ra} T_{ra} \!+\! \dot{m}_{oa} T_{oa}}{\dot{m}_{sa}}\,,\\
&\dot{m}_{ra} \!+\! \dot{m}_{oa}\!=\!\dot{m}_{sa}\,.
}
$\dot{m}_{ra}$ denotes the mass air flow rate of the recirculated air}. The thermal power provided by the AHU, denoted $\dot{Q}_{ahu}$, is given by:
\mysubeq{eq:Q ahu constraints}{
\dot{Q}_{ahu} &= c_p\dot{m}_{sa}(T_{sa} - T_{ma})\\
\dot{Q}_{ahu} &\in [-Q_{c}^{rated},Q_{b}^{rated}]\,,
}
where $T_{sa}$ is the temperature of the supply air to the building zones. A positive value for $\dot{Q}_{ahu}$ indicates that the AHU heating coil was utilized to regulate the temperature of the supply air. Similarly, when the cooling coil is utilized to regulate the duct air temperature $\dot{Q}_{ahu} < 0$. 

The thermal power delivered by the VAV {terminal} box in each zone-$i$, the boiler and the chiller, denoted by $\dot{Q}_{reheat}^{i},\,\dot{Q}_{b}$ and $\dot{Q}_{c}$, respectively, are given by:
\begin{subequations}
\begin{align}
\dot{Q}_{reheat}^{i} &= c_p \dot{m}_{sa}^{i}(T_{da}^{i} - T_{sa})\label{eq:Q reheat constraints}\\
\dot{Q}_{b} &= {\sum}_{i=1}^N\dot{Q}_{reheat}^{i} + \max\lbrace 0,\dot{Q}_{ahu}\rbrace\label{eq:Qb}\\
\text{and }\dot{Q}_{e} &= \max\lbrace 0,-\dot{Q}_{ahu}\rbrace\,. \label{eq:Qc}
\end{align}\end{subequations}
Denoting the rated values of the boiler and chiller capacities by $\dot{Q}_b^{rated}$ and $\dot{Q}_e^{rated}$, respectively, the following physical
constraints must hold:
\begin{subequations}\label{eq:boil_chill}
\begin{align}
&\dot{Q}_b\in[0,\dot{Q}_b^{rated}]\,,~\dot{Q}_e\in[0,\dot{Q}_e^{rated}]\\
\text{and }&\dot{Q}_{ahu}\in[-\dot{Q}_e^{rated},\dot{Q}_b^{rated}]\,.
\end{align}\end{subequations}

Finally, the efficiency of the HVAC system is determined by the power consumed in the supply fan, the boiler and the chiller. The electrical power used by the supply fan is determined by the mass flow rate of the supply air as
{\begin{subequations}\label{eq:P fan constraints} \begin{align}
P_{fan} &= \dfrac{\Delta P}{\eta_{tot} \rho_{air}} \dot{m}_{design} f_{pl}\\
f_{pl} &= {\sum}_{k=1}^4 c_{f,k}\,f_{flow}^{k-1}\,,~f_{flow} = \dot{m}_{sa} / {\dot{m}_{design}}
\end{align}\end{subequations}
where $f_{flow}$ represents the flow fraction or part-load ratio, $c_{f,k}\,,\,k\!\in\!\lbrace 1,2,3,4\rbrace$, are parameters that capture the fan efficiency, $\eta_{tot}$ represents fan total efficiency, $\Delta P$ represents fan design pressure increase, $\rho_{air}$ represents air density, and  $\dot{m}_{design}$ is the design (maximum) air flow. The power consumed by the boiler is given by:
{\begin{subequations}\label{eq:P boiler constraints} \begin{align}
P_{boiler} &\!=\! \dfrac{\dot{Q}_{b}}{\eta_{thermal}\, \eta_{eff}}\\
\eta_{eff} &\!=\! {\sum}_{k=1}^3c_{b,k}\,{PLR_{b}}^{k\!-\!1},~PLR_{b} \!=\! {\dot{Q}_{b}} / {\dot{Q}_{b}^{rated}}\!
\end{align}\end{subequations}
where 
$PLR_{b}$ denotes the boiler part-load ratio, $\eta_{eff}$ represents boiler efficiency curve, $\eta_{thermal}$ represents nominal thermal efficiency, and the coefficients $c_{b,k}\,,\,k\!\in\!\lbrace 1,2,3\rbrace,$ describe the boiler performance. The power consumption of the chiller is given by:
\begin{subequations}\label{eq:P chiller constraints}\begin{align}
\!\!P_{chiller} &\!=\! f_{gen} \dot{Q}_{e} + P_{pump} \\
f_{gen} &\!=\! {\sum}_{k=1}^3c_{g,k}\,{PLR_{e}}^{k\!-\!2}, ~PLR_{e} \!=\! \dot{Q}_{e} / \dot{Q}_{e}^{rated}\!
\end{align}\end{subequations}
where 
$PLR_{c}$ denotes the chiller part load ratio, $f_{gen}$ represents generator heat input ratio,$P_{pump}$ denotes the pump power, and the coefficients $c_{g,k}\,\,k\in\lbrace 1,2,3\rbrace,$ describe the chiller efficiency.
The nominal values for the parameters characterizing the fan, boiler, and chiller models used in this study are as follows: number of zones $N$ = 5; $c_p$ = 1005 J/(kg-K); (\textit{Fan}) $\Delta P$ = 1000 kg/(m-s)$^2$, $\eta_{tot}$ = 0.7, $\rho_{air}$ = 1.225 kg/m$^3$, $\dot{m}_{design}$ = 2.98 kg/s, $c_{f,1}$ = 0.3507, $c_{f,2}$ = 0.3085, $c_{f,3}$ = -0.5413, $c_{f,4}$ = 0.8719; (\textit{Boiler}) $Q_{b}^{rated}$ = 1.09$\times$10$^{8}$ J/hr, $\eta_{thermal}$ = 0.8, $c_{b,1}$ = 0.97, $c_{b,2}$ = 0.0633, $c_{b,3}$ = -0.0333; (\textit{Chiller}) $Q_{e}^{rated}$ = $1.47\times 10^8$ J/hr, $P_{pump}$ = 1.8$\times$10$^6$ J/hr, $c_{g,1}$ = 0.03303, $c_{g,2}$ = 0.6852, $c_{g,3}$ = 0.2818.

\subsection{Optimal Baseline Determination}
We select as \emph{baseline} the operational scenario corresponding to the minimal hourly energy consumption profile, estimated under the assumption of perfect system model parameters, sensors information and internal and external load forecast. The buildings energy management systems can then compute the set-point values for equipment controllers to enable most efficient operation of the building HVAC system. Any flexibility that the buildings may provide to the grid operations are measured as changes to this baseline operation. Accurate determination of the baseline power consumption profile is, therefore, critical for successful grid integration of responsive buildings. 

As shown in \cite{Ramachandran:2017}, the problem of determining optimal operational baseline can be formulated as a constrained nonlinear optimization. The supervisory set-points which serve as the decision variables for optimization problem are: {$\dot{m}_{oa}$\,, $\dot{m}_{ra}$\,, $T_{sa}$\,, and $T_{da}^{i}$\,, $\dot{m}_{sa}^{i}$ (for each building zone $i$).} The decision variables and the associated limits of their values for the specific use case are presented in Table\,\ref{tab:dv}. The optimization problem is given by \cite{Ramachandran:2017}:
\begin{align*} 
\underset{T_{sa},\dot{m}_{oa},\dot{m}_{ra},T_{da}^{i},\dot{m}_{sa}^{i}}{\text{minimize}}&~ 
\alpha_{el}\left(P_{fan} \!+\! P_{chiller}\right) \!+\! \alpha_{ng}\,P_{boiler}\,,\\
\text{subject to:}\quad&~\text{constraints in \eqref{eq:Q zone constraints}-\eqref{eq:boil_chill} and Table\,\ref{tab:dv}\,.}
\end{align*}
where $\alpha_{el}\!=\!3.167$ and $\alpha_{ng}\!=\!1.084$ are the \textit{source conversion factors} associated with electricity (which drives the fan and the chiller) and natural gas (which drives the boiler), respectively, and refer to the actual energy need to be purchased in order to drive the HVAC system. In order to solve the problem, it is assumed that the chiller, boiler and fan parameters are known; hourly thermal load profiles for each zone ($\dot{Q}_{zone}^i$) have been generated using EnergyPlus \cite{Crawley:2000}; and that outside air temperature forecast is available. The optimization problem is scripted in MATLAB and solved via IPOPT \cite{Wachter:2006}.

\begin{table*}[t]
\centering
\caption{Decision Variables: notation and associated constraints}\label{tab:dv}
\begin{tabular}{ | c | c | c | c |}
\hline
Variable & Notation & Lower Limit & Upper Limit \\ \hline
supply air temperature & $T_{sa}$ & $12^{\circ}$C & $37^{\circ}$C \\ \hline
outside air mass flow rate & $\dot{m}_{oa}$ & $\sum_{i = 1}^N \dot{m}_{oa}^{i,min}$ & $\dot{m}_{design}$ \\ \hline
recirculated air mass flow rate & $\dot{m}_{ra}$ & 0 & $\dot{m}_{design}$ \\ \hline
discharge air temperature in building zone $i$ & $\dot{T}_{da}^i$ & $T_{sa}$ & $37^{\circ}$C \\ \hline
supply air mass flow rate in building zone $i$ & $\dot{m}_{sa}^i$ & $\sum_{i = 1}^N m_{sa}^{i,min}$ & $\dot{m}_{design}$ \\ \hline
\end{tabular} 
\end{table*}

\section{Numerical Example}
\label{sec:Example}

We utilize the approach described in Section \ref{sec:Methodology} the sensitivity of a computed baseline profile to uncertainties in model parameters and input values. Specifically in this paper, we consider the following set of parameters and input values for uncertainty analysis: 1) the outside air temperature ($T_{oa}$), and 2) the fan efficiency parameters $c_{f,k}\,\forall k=1,2,3,4$\,. We are interested in analyzing the impact of errors in estimated values of (possibly a subset of) $w$ on the optimal value of baseline power consumption. We model the deviations as a fraction of the nominal values, i.e.
\begin{align*}
\Delta=\alpha\,\left|w_0\right|,
\end{align*}
for some chosen $\alpha$\,. Furthermore, we are interested in exploring which of the input data and parameters the optimal baseline value is most sensitive to, and how this sensitivity may change under different operating scenarios.


%
\begin{figure*}[thpb]
\centering
\captionsetup{justification=centering}
\subfigure[hot day]{
\includegraphics[scale=0.317]{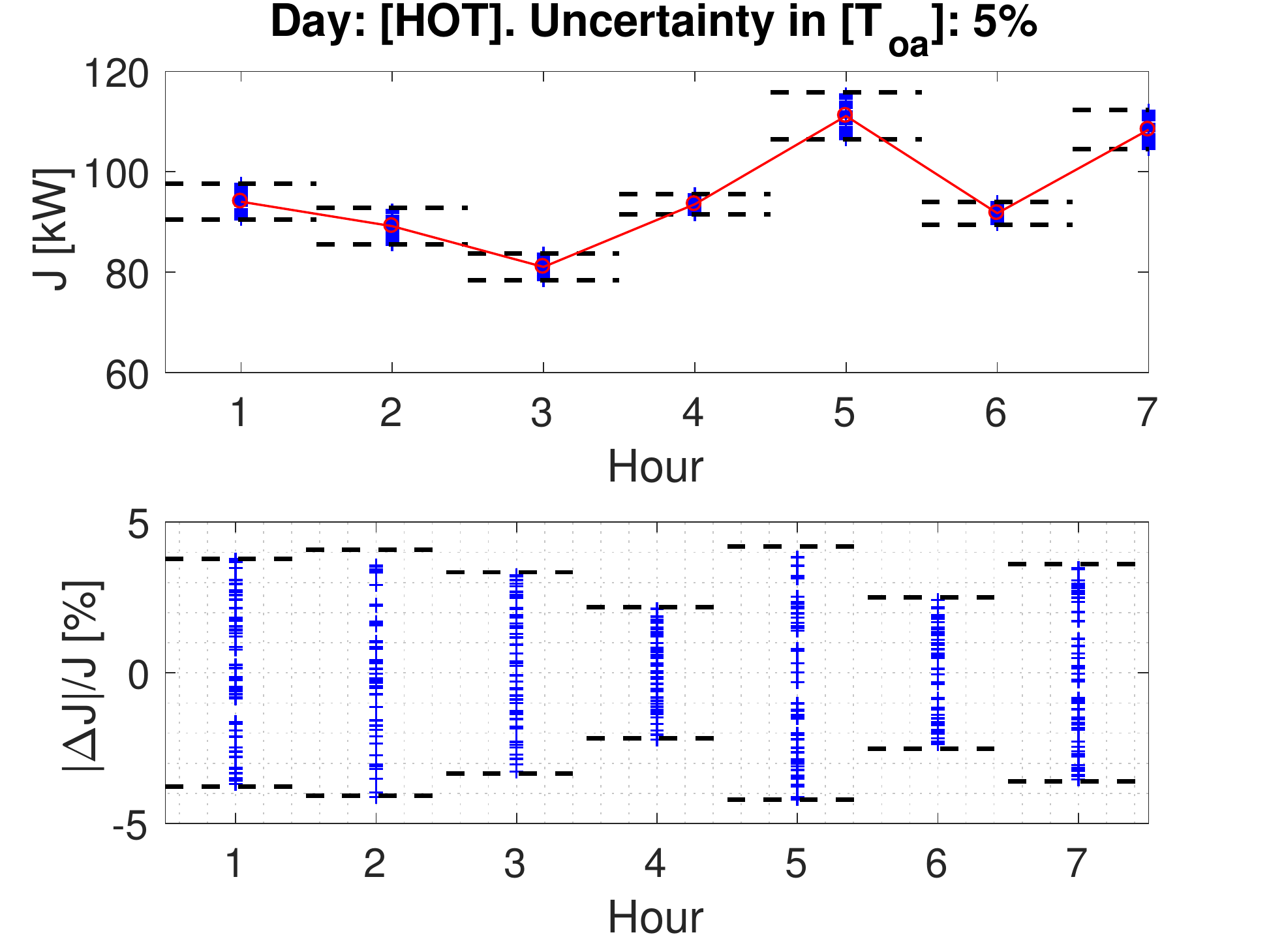}\label{F:toa_hot_5}
}
\hspace{-0.37in}
\subfigure[moderate day]{
\includegraphics[scale=0.317]{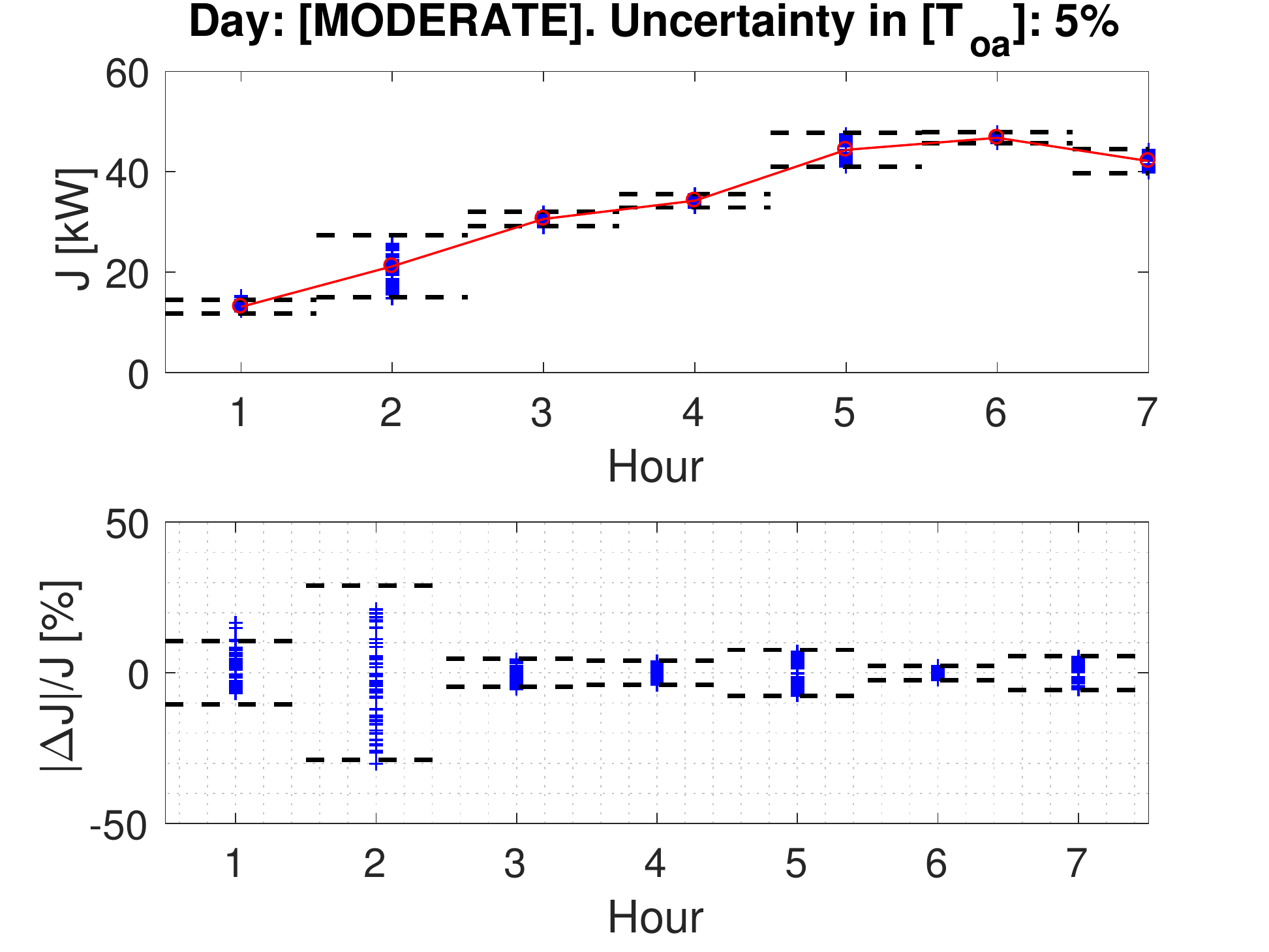}\label{F:toa_mod_5}
}
\hspace{-0.37in}
\subfigure[cold day]{
\includegraphics[scale=0.317]{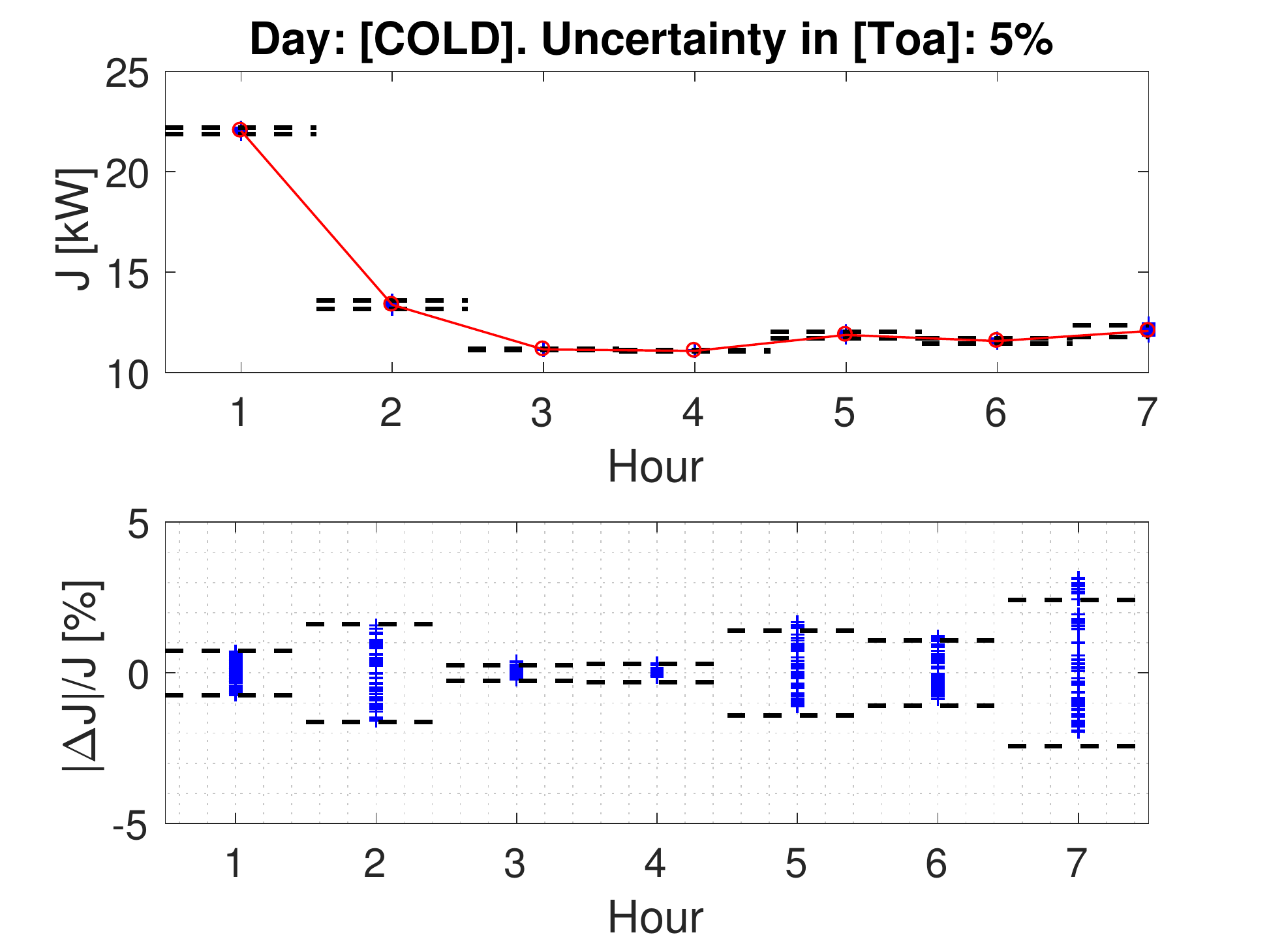}\label{F:toa_cold_5}
}
\caption[Optional caption for list of figures]{Baseline power consumption profiles for a commercial HVAC system for hot, cold and moderate days, and its changes (absolute as well as relative) with respect to $\pm 5\%$ change in outside air temperature.}
\label{F:toa_5}
\end{figure*}

To facilitate the analysis, we use EnergyPlus to simulate the building HVAC system models for typical load profiles for seven hours during the building occupied period, and for three external load profiles characterized as hot, cold and moderate. In Fig.\,\ref{F:toa_5} we present the results of the sensitivity analysis in the form of changes to optimal cost function, when the outside air temperature has $\pm 5\%$ uncertainty (around nominal). We present both the absolute as well as relative changes in the baseline power consumption profile, at an hourly basis for the three chosen days. As expected, the hot day experiences higher HVAC power consumption. Interestingly, the sensitivity of the optimal solution is highest on the moderate day. This is in line with the observation made in \cite{Ramachandran:2017} that when the outside temperature is moderate the HVAC system operates in an \textit{economizing} mode bringing in a lot of outside air, which is why the operation becomes very sensitive to the outside temperature. In both hot and cold day, the air-flow from outside is usually lowest rendering the operation less sensitive to changes in outside temperature.

\begin{figure*}[thpb]
\centering
\captionsetup{justification=centering}
\subfigure[hot day]{
\includegraphics[scale=0.317]{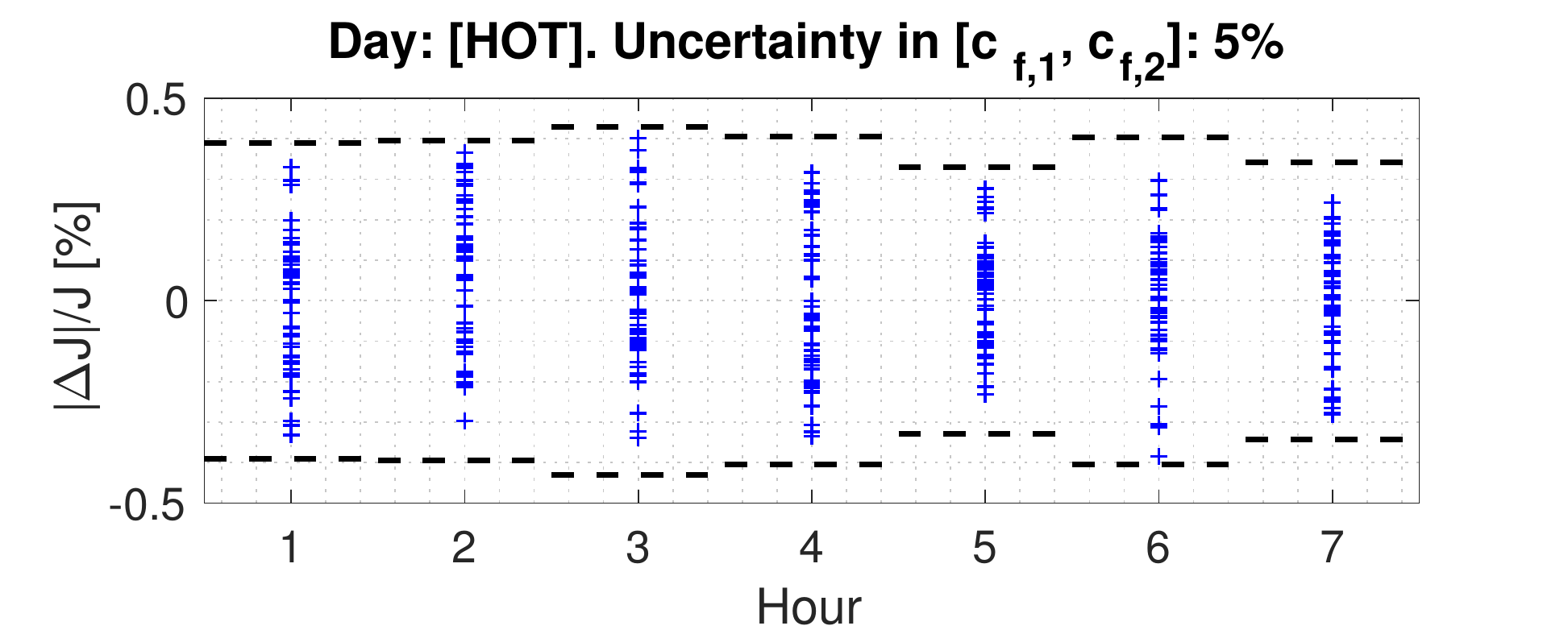}\label{F:fan_hot_5}
}
\hspace{-0.37in}
\subfigure[moderate day]{
\includegraphics[scale=0.317]{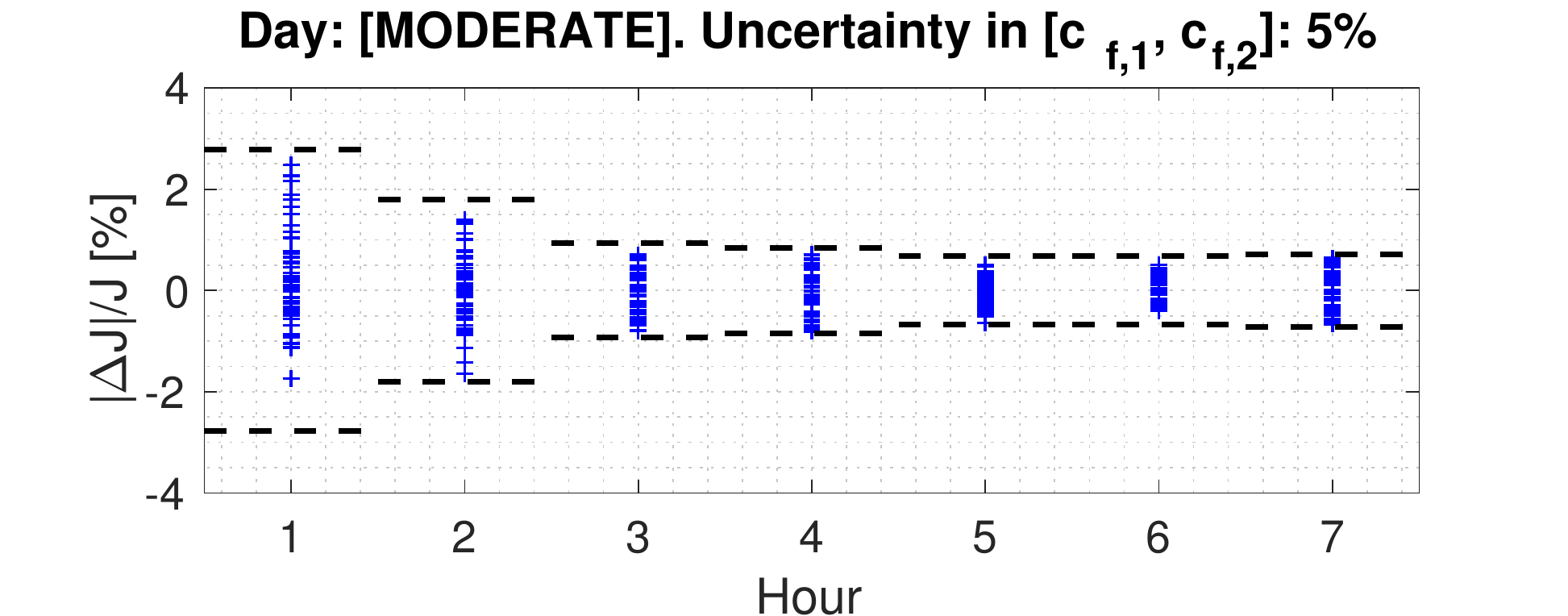}\label{F:fan_mod_5}
}
\hspace{-0.37in}
\subfigure[cold day]{
\includegraphics[scale=0.317]{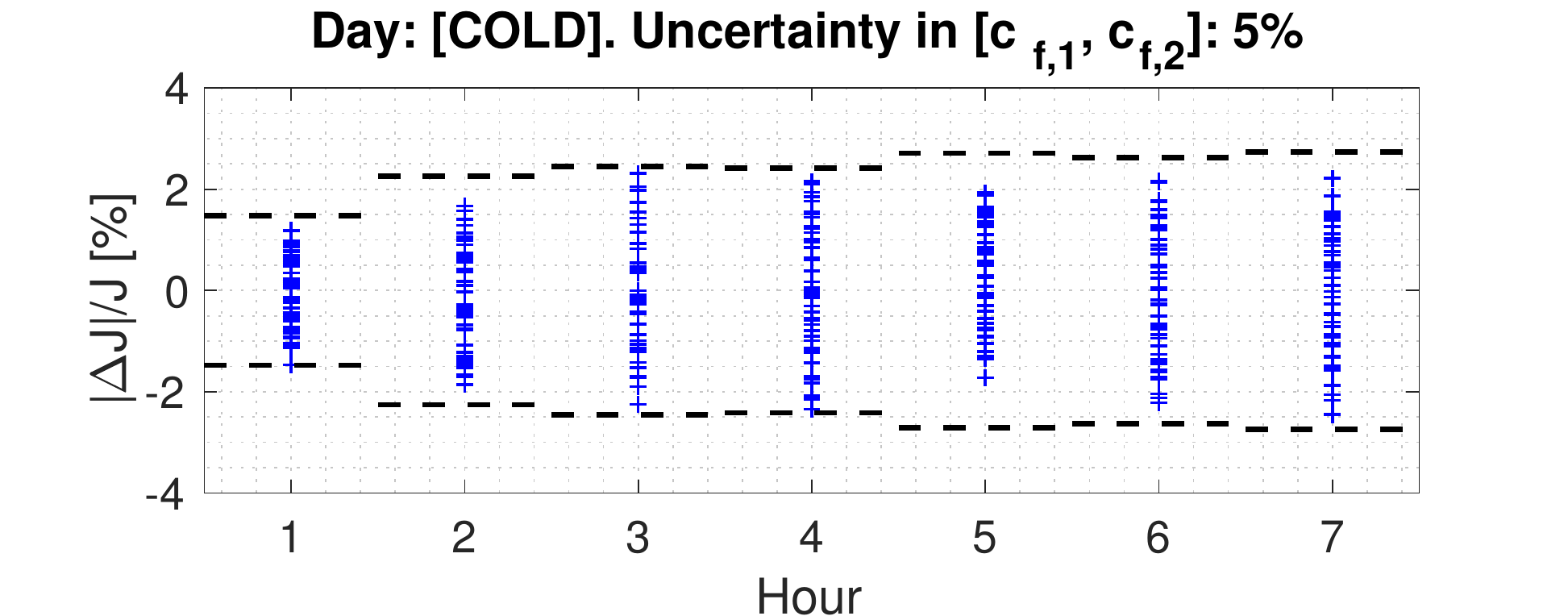}\label{F:fan_cold_5}
}
\caption[Optional caption for list of figures]{Sensitivity of baseline power consumption to $\pm 5\%$ change in fan efficiency parameters.}
\label{F:fan_5}
\end{figure*}
Next we present the sensitivity analysis to fan efficiency parameters in Fig.\,\ref{F:fan_5}. It can be noted that the sensitivity of the baseline hourly optimal energy consumption to uncertainty in the fan parameters is lowest during the hot day, with less than $\pm$0.5\% relative change in optimal cost.  We observe that during a hot day the chiller would be consuming most of the total system power consumption making the contribution of fan power relatively small compared with the baseline power consumption. The sensitivity increases to about $\pm$3\% on cold day, and during some hours in a day with moderate moderate values for the outdoor air temperature.


\section{Conclusion}
We developed a methodology for formally quantifying the sensitivity of optimal hourly energy consumption, a baseline energy consumption model, to various sources of parametric and measurement uncertainties. We then demonstrate the method for a use case building HVAC example.
This approach can be utilized to provide hourly energy consumption forecasts, and most importantly the associated uncertainty around those forecasts, to resource aggregators. This will enable participation of building energy systems, typical energy consumers, to grid operation and grid service markets in a risk aware context.


\section*{Acknowledgment}

This work has been supported by the the Buildings Technologies Office of the U.S. Department of Energy’s Office of Energy Efficiency and Renewable Energy. 

\IEEEtriggeratref{17}


\bibliographystyle{IEEEtran}
\bibliography{IEEEabrv,mybibfile,RefKundu,RefList,references}
%


\end{document}